 \documentclass[12pt]{amsart}
 \usepackage[dvips]{graphicx}
 \usepackage{epstopdf}
 \usepackage{graphicx}
 \usepackage{amscd,amsmath,amsthm,amssymb}
 \usepackage{pstcol,pst-plot,pst-3d}
\newpsstyle{fatline}{linewidth=1.5pt}
\newpsstyle{fyp}{fillstyle=solid,fillcolor=verylight}
\definecolor{verylight}{gray}{0.97}
\definecolor{light}{gray}{0.93}
\definecolor{medium}{gray}{0.82}

\usepackage{amsfonts,amssymb,amscd,amsmath,enumerate,verbatim}
 \def\NZQ{\mathbb}               
 \def\NN{{\NZQ N}}
 
 \def\ZZ{{\NZQ Z}}

 \def\PP{{\NZQ P}}

 \def\frk{\mathfrak}               

 \def\mm{{\frk m}}

 \def\G{{\mathcal G}}

 \def\ab{{\mathbf a}}
 \def\bb{{\mathbf b}}

 \def\cb{{\mathbf c}}

 \def\opn#1#2{\def#1{\operatorname{#2}}} 
 \opn\chara{char} \opn\length{\ell} \opn\pd{pd} \opn\rk{rk}
 \opn\projdim{proj\,dim} \opn\injdim{inj\,dim} \opn\rank{rank}
 \opn\depth{depth} \opn\grade{grade} \opn\height{height}
 \opn\embdim{emb\,dim} \opn\codim{codim}
   \opn\Ap{Ap}

 \opn\Tr{Tr} \opn\bigrank{big\,rank}
 \opn\superheight{superheight}\opn\lcm{lcm}
 \opn\trdeg{tr\,deg}
 \opn\reg{reg} \opn\lreg{lreg} \opn\ini{in} \opn\lpd{lpd}
 \opn\size{size} \opn\sdepth{sdepth}
 \opn\link{link}\opn\fdepth{fdepth}\opn\lex{lex}
 \opn\tr{tr}
 \opn\type{type}
 \opn\div{div} \opn\Div{Div} \opn\cl{cl} \opn\Cl{Cl}
 \opn\Spec{Spec} \opn\Supp{Supp} \opn\supp{supp} \opn\Sing{Sing}
 \opn\Ass{Ass} \opn\Min{Min}\opn\Mon{Mon}
 \opn\Ann{Ann} \opn\Rad{Rad} \opn\Soc{Soc}
 \opn\Im{Im} \opn\Ker{Ker} \opn\Coker{Coker} \opn\Am{Am}
 \opn\Hom{Hom} \opn\Tor{Tor} \opn\Ext{Ext} \opn\End{End}
 \opn\Aut{Aut} \opn\id{id}
 
 \opn\nat{nat}
 \opn\pff{pf}
 \opn\Pf{Pf} \opn\GL{GL} \opn\SL{SL} \opn\mod{mod} \opn\ord{ord}
 \opn\Gin{Gin} \opn\Hilb{Hilb}\opn\sort{sort}
 \opn\PF{PF}\opn\Ap{Ap}
 \opn\aff{aff} \opn
 \con{conv} \opn\relint{relint} \opn\st{st}
 \opn\lk{lk} \opn\cn{cn} \opn\core{core} \opn\vol{vol}  \opn\inp{inp} \opn\nilpot{nilpot}
 \opn\link{link} \opn\star{star}\opn\lex{lex}\opn\set{set}
 \opn\width{wd}
 \opn\Fr{F}
 \opn\QF{QF}
 \opn\G{G}
 \opn\type{type}\opn\res{res}
 \opn\gr{gr}
 
 \def\pot#1#2{#1[\kern-0.28ex[#2]\kern-0.28ex]}

 \opn\dirlim{\underrightarrow{\lim}}
 \opn\inivlim{\underleftarrow{\lim}}
 \let\union=\cup
 \let\sect=\cap
 \let\dirsum=\oplus
 
 \let\iso=\cong
 \let\Union=\bigcup
 
 \let\Dirsum=\bigoplus
 
 \let\to=\rightarrow
 
 \def\Implies{\ifmmode\Longrightarrow \else
         \unskip${}\Longrightarrow{}$\ignorespaces\fi}
 \def\implies{\ifmmode\Rightarrow \else
         \unskip${}\Rightarrow{}$\ignorespaces\fi}
 \def\iff{\ifmmode\Longleftrightarrow \else
         \unskip${}\Longleftrightarrow{}$\ignorespaces\fi}

 \let\:=\colon
 \newtheorem{Theorem}{Theorem}[section]
 \newtheorem{Lemma}[Theorem]{Lemma}
 \newtheorem{Corollary}[Theorem]{Corollary}
 \newtheorem{Proposition}[Theorem]{Proposition}

 \newtheorem{Example}[Theorem]{Example}
 \newtheorem{Examples}[Theorem]{Examples}

 \let\epsilon\varepsilon
 \let\kappa=\varkappa
 \def\qed{\ifhmode\textqed\fi
       \ifmmode\ifinner\quad\qedsymbol\else\dispqed\fi\fi}
 \def\textqed{\unskip\nobreak\penalty50
        \hskip2em\hbox{}\nobreak\hfil\qedsymbol
        \parfillskip=0pt \finalhyphendemerits=0}
 \def\dispqed{\rlap{\qquad\qedsymbol}}
 
 \opn\dis{dis}
 \def\pnt{{\raise0.5mm\hbox{\large\bf.}}}
 
 \opn\Lex{Lex}

\begin{document}
\title {The fiber cone of a monomial ideal in two variables}

\author {J\"{u}rgen  Herzog, Ayesha Asloob Qureshi and Maryam Mohammadi Saem}

\address{J\"urgen Herzog, Fachbereich Mathematik, Universit\"at Duisburg-Essen, Campus Essen, 45117
Essen, Germany} \email{juergen.herzog@uni-essen.de}

\address{Ayesha Asloob Qureshi, Sabanc\i \; University, Faculty of Engineering and Natural Sciences, Orta Mahalle, Tuzla 34956, Istanbul, Turkey}
\email{aqureshi@sabanciuniv.edu}

\address{Maryam Mohammadi Saem,  Faculty of Science, University of Mohaghegh Ardabili, P.O. Box 179, Ardabil, Iran}
\email{m.mohammadisaem@yahoo.com}

\begin{abstract}
We determine in an explicit way  the depth of the fiber cone and its relation ideal for classes of monomial ideals in two variables. These classes include concave and convex ideals as well as symmetric ideals.
\end{abstract}

\thanks{}




\maketitle

\setcounter{tocdepth}{1}

\section*{Introduction}

The study of the number of generators of the powers of graded ideals in the polynomial ring $S=K[x_1,\ldots,x_n]$ leads naturally to the study of the Hilbert function of the fiber cone of these ideals. Indeed, if $I\subset S$ is a graded ideal, then $\mu(I^k)=\dim_KF(I)_k$, where $F(I)_k$ is the $k$th graded component of the fiber cone $F(I)$ of $I$. Recall that $F(I)=R(I)/\mm R(I)$, where $R(I)=\Dirsum_{k\geq 0}I^k$ is the Rees ring of $I$ and $\mm=(x_1,\ldots,x_n)$ is the graded maximal ideal of $S$. It can be easily seen that $\mu(I^k)< \mu(I^{k+1})$ for all $k\geq 1$, if $\depth F(I)>0$ and $I$  is not a principal ideal. If it happens that all generators of $I$ are of same degree, say $I=(f_1,\ldots,f_m)$ with $\deg f_i=d$ for all $i$, then $F(I)$ can be identified with $K[f_1,\ldots,f_m]\subset S$, and hence in this case $F(I)$ is a domain. Thus we see that for a non-principal ideal $I$ the inequality $\mu(I^k)\geq  \mu(I^{k+1})$ for some $k$  is only possible if $\depth F(I)=0$. Thus it is of interest to study the depth of $F(I)$. Of particular interest are the extreme cases, namely when $\depth F(I)=0$ or when $\depth F(I)=\dim F(I)$, which is the maximal possible and in which case $F(I)$ is Cohen--Macaulay.

In this paper we restrict ourselves to monomial ideals $I\subset K[x,y]$. Removing a possible common factor of the generators we may assume that their greatest common divisor is one. This does not affect the number of the generators of the powers of the ideal. With this assumption on the generators, the unique minimal monomial set of generators of $I$, denoted $G(I)$,  determines and is determined by two sequences of integers
\[
\ab: a_1>a_2>\ldots >a_m=0 \quad \text{and} \quad \bb: 0=b_1< b_2< \ldots <b_m.
\]
Indeed, if the set of monomials $\mathcal{S}= \{u_1,u_2,\ldots, u_m\}$ is a set of monomial generators of $I$, and if $u_i=x^{a_i}y^{b_i}$ for $i=1,\ldots,m$,   and furthermore $\gcd(u_1,u_2,\ldots, u_m)=1$ and $u_1>u_2>\ldots > u_m$ with respect to the lexicographical order,  then $\mathcal{S}$ is the unique monomial set of generators of $I$  if and only if  the corresponding exponent sequences of the $u_i$  satisfy the above inequalities.

The fiber cone $F(I)$ of such $I\subset K[x,y]$ is of dimension 2, and hence the depth of $F(I)$ can be $0$, $1$ or $2$. It is a big challenge to determine the depth of $F(I)$ in terms of  the sequences $\ab$ and $\bb$ in an explicit way.

This problem has been studied in numerous papers, but only in the case that all generators of $I$ have the same degree. In this case,  $F(I)$ may be considered as the homogeneous coordinate ring of the projective monomial curve  defined by the numerical semigroup $H$ generated by the integers $a_1,\ldots,a_m$. For such rings the depth can be only $1$ or $2$.

Probably the first paper dealing with the homogeneous coordinate ring of a projective monomial curve is the paper of Bresinsky and Renschuch \cite{BR} in which they study  projective monomial curves in $\PP^3$ and describe a minimal set of generators for the defining ideal of the homogeneous coordinate ring  of these curves. Note that such a ring is simply the fiber cone  of an equigenerated  monomial ideal $I\subset K[x,y]$ with $\mu(I)=4$. This is the first non-trivial case to be considered and it turns out that even this case is not so easy  to deal with. There exist several nice and interesting criteria for a projective monomial curve in $\PP^n$ to be arithmetically Cohen--Macaulay, see for example the paper of Cavaliere and Niesi \cite{CN}, the paper by Reid and Roberts \cite{RR} and that of  Molinelli, Patil and Tamone \cite{MPT}. A Gr\"obner basis criterion is given by Kamoi \cite{Ka}. For monomial curves in $\PP^3$ there is a very interesting result of Bresinsky, Schenzel and Vogel \cite{BSV} which says that such a curve is arithmetically Cohen--Macaulay if and only if the defining ideal of  its coordinate ring is generated by at most $3$ elements. Then,  by using the above mentioned result of  Bresinsky and Renschuch,   they show that the projective curve associated with the numerical semigroup with generators $a<b<c=a+b$ is arithmetically Cohen--Macaulay if and only if $b=a+1$. In Theorem~\ref{sum} we give an alternative proof of the theorem, not using the structure theorem of Bresinsky and Renschuch, but instead use the Cohen--Macaulay criterion of Cavaliere and Niesi. We also extend this result to symmetric ideals as described below.

In this paper, however,  we focus on the case that $I\subset F(I)$ is not equigenerated, in which case the sequence $\ab$ and the sequence $\bb$ do not determine each other, $F(I)$ is not the homogeneous coordinate  ring of a monomial curve and the depth of $F(I)$ may very well be equal to $0$. The graded contracted ideals $I\subset K[x,y]$  which include the normal ideals are examples for which $F(I)$ is Cohen--Macaulay,  see \cite{CNR}. In this paper two classes of monomial ideals in $K[x,y]$ are considered, the concave and convex ideals as well as the  symmetric ideals. We call $I\subset K[x,y]$ concave, resp.\ convex if the sequence $\cb_1,\ldots,\cb_m$ of the exponent vectors of $G(I)$ form a concave resp.\ convex sequence, which means that $2\cb_i\geq \cb_{i-1}+\cb_{i+1}$, resp.\ $2\cb_i\leq \cb_{i-1}+\cb_{i+1}$ for $i=2,\ldots, m-1$. We say that $I$ has an inner corner point, if for some index  $i$, the corresponding inequality is strict. This class of ideals and their fiber cone are studied in Section 2, which follows Section~1, where some basic  and simple facts about concave and convex sequences are collected. The main results in Section 2 are Theorem~\ref{concave1}  and Theorem~\ref{convex2},  where it is shown that the  fiber cone of concave and convex ideals are  Cohen--Macaulay Koszul algebras. In both cases, the generators of the defining ideal $L$  of their fiber cone,  are explicitly determined in terms of the exponent sequence, and it is shown that $L$ admits a quadratic Gr\"obner basis. Besides of these common properties,  the concave and convex ideals differ in many ways. While the fiber cone of convex ideals is radical, this is not the case for those  concave ideals which admit an inner corner point, and while all powers of a convex ideal are again convex, proper powers of  concave ideals admitting an inner corner point are never concave, see Proposition~\ref{concavepowers} and Proposition~\ref{convexpowers}. Let $J\subset I$ be the ideal generated by the pure powers of $x$ and $y$. It is shown in Proposition ~\ref{concavehilbert} that for a concave ideal  the reduction number of $I$ with respect to $J$ is one, while for a convex ideal admitting an inner corner point, $J$ is never a reduction ideal of $I$, see Proposition~\ref{convexpowers1}.

The function which is composed by the line segments in $\ZZ^2$  connecting the exponent vectors of a concave ideal is a concave function. Conversely one could choose the exponent vectors of an ideal $I\subset K[x,y]$ on a given concave function connecting the $x$-axis with the $y$-axis. This more general class of ideals however does not have such nice properties as our concave ideals. Their fiber ring may not be Cohen--Macaulay and its  defining ideal will in general not be defined in degree 2.

Section~3 is devoted to the study of symmetric ideals which are  the ideals whose $\bb$-sequence is  just the $\ab$-sequence  in reverse order. In this way, the number of parameters defining the ideal is halved. In Proposition~\ref{shalom} we consider for each $m\geq 5$ a symmetric ideal $I$, first studied in \cite{EHM},  and show that  $\depth F(I)=0$. We do not know of any symmetric ideal  which is generated by less than 5 elements and whose fiber cone has depth $0$. The fiber cone  of a symmetric ideal generated by 2 elements is a 2-dimensional polynomial ring, and for a 4-generated symmetric ideal it is a 2-dimensional hypersurface ring. So $m=4$ is the smallest number for which the fiber cone of an $m$-generated symmetric ideal may have depth $0$.

A symmetric ideal generated by 4 elements is given by a sequence of three integers $0<a<b<c$. The corresponding symmetric ideal is $I=(x^c, x^by^a,x^ay^b, y^c)$. By using the results of Section 2 it is shown in Theorem~\ref{application}  that $F(I)$ is Cohen--Macaulay,  if  $2a \leq b$ and $2b \leq a+c$, or $2a \geq b$ and $2b \geq a+c$. The ideal $I$ is equigenerated if and only if $c=a+b$. As mentioned above, in Theorem~\ref{sum} we recover the result of Bresinsky, Schenzel and Vogel \cite{BSV} which says that $F(I)$ is Cohen--Macaulay if and only of $b=a+1$. In this case, when $b=a+1$, it is shown in Theorem~\ref{reduction} that $J=(x^c,y^c)$ is a reduction ideal of $I$ and the reduction number of $I$ with respect to $J$ is $a$. If $c\neq a+b$, then $F(I)$ is no longer a domain. In Theorem~\ref{sum}, which is the main result of this section, we show that for `large' and `small' $c$, the fiber cone $F(I)$ of $I$ is Cohen--Macaulay. This fact is summed up in Corollary~\ref{interval},   where it is stated that $F(I)$ is Cohen--Macaulay, if $c$ does not belong to the interval $[2a+1,r(b-a)+a]$ with $r=\lceil b/(b-a)\rceil$. Together with Theorem~\ref{sum}, this has the nice consequence that $F(I)$ is Cohen--Macaulay for all $c$, if $b=a+1$, see Corollary~\ref{nice}. Another consequence (Corollary~\ref{shift}) is that for any given sequence $0<a<b<c$ the fiber cone of the corresponding  symmetric ideal is Cohen-Macaulay for any shifted sequence $0<a+m<b+m<c+m$ with  $m\geq c-2a$. For us the terra incognita is the interval  $[2a+1,r(b-a)+a]$,  where for $c$ belonging to this interval, the depth of the fiber of  the corresponding symmetric ideal  may be one or two, but never zero in our examples.

\section{Concave and convex sequences of integer vectors in $\ZZ^2$}
\label{sequences}

For any two vectors $\ab, \bb \in \ZZ^n_{\geq 0}$ we set  $\ab\leq \bb$,  if this inequality is valid componentwise, and we set $\ab<\bb$, if $\ab\leq \bb$ and $\ab\neq \bb$. A sequence $\mathcal{A}$ of integer vectors $\ab_1,\ldots,\ab_m$ in $\ZZ_{\geq 0}^2$ is called {\em convex } resp.\ {\em concave}, if $2\ab_i\leq \ab_{i-1}+\ab_{i+1}$ resp.\ $2\ab_i\geq \ab_{i-1}+\ab_{i+1}$ for $i=2,\ldots,m-1$.

Let $\mathcal{A}$ be a  convex, resp.\  concave sequence. We call $\ab_i$ a {\em corner point} of  $\mathcal{A}$,  if $i=1$ or $i=m$, or if $2\ab_i< \ab_{i-1}+ \ab_{i+1}$, resp.\ $2\ab_i> \ab_{i-1}+ \ab_{i+1}$.

\medskip
The following inequalities will be used later.

\begin{Lemma}\label{seq}
{\em (a)} Let $\mathcal{A}\: \ab_1,\ldots,\ab_m$ be a concave sequence of vectors in $\mathcal \ZZ^2$. Then
\begin{eqnarray}
\label{inequality1}
\ab_i+\ab_j\geq \ab_{i-k}+\ab_{j+k},
\end{eqnarray}
for all $i\leq j$ and all $k$ such that $1\leq i-k$ and $j+k\leq m$.

{\em (b)} Let $\ab_1,\ldots,\ab_m$ be a convex sequence of vectors in $\mathcal \ZZ^2$. Then
\begin{eqnarray*}
\label{inequality2}
\ab_i+\ab_j \leq \ab_{i-k}+\ab_{j+k},
\end{eqnarray*}
 for all $i\leq j$ and all $k$ such that $1\leq i-k$ and $j+k\leq m$.

{\em (c)} The inequalities {\em (\ref{inequality1})} and {\em (\ref{inequality2})} are strict, if there exists an integer $r$ with $i<r<j$ and such that $\ab_r$ is a corner point of the sequence $\mathcal{A}$.
\end{Lemma}

\begin{proof}
(a) It is enough to prove the inequality
\begin{equation}\label{ineq}
\ab_i+\ab_j\geq \ab_{i-1}+\ab_{j+1},
\end{equation}
for all $i\leq j$ and $1\leq i-1$ and $j+1\leq m$, because (\ref{inequality1})  follows then  by the repeated application of inequality (\ref{ineq}). To prove inequality~(\ref{ineq}), we apply induction on $l=j-i$. If $l=0$, the assertion follows from the definition of concave sequences. Assume that (\ref{ineq}) holds for all $k <l$. Again, by using definition of concave sequences, we have the following
\begin{equation}\label{ineq1}
2(\ab_i+\ab_j)\geq \ab_{i-1}+\ab_{i+1}+\ab_{j-1}+\ab_{j+1}
\end{equation}

Note that $j-1-(i+1)=j-i-2<l$ and that $i+1\leq j-1$. So we can use inequality~(\ref{ineq}) and obtain,
\[
\ab_{i+1}+\ab_{j-1} \geq \ab_{i}+\ab_{j}
\]
By using this  inequality together with (\ref{ineq1}), we get
\[
2(\ab_i+\ab_j)\geq \ab_{i-1}+\ab_{i}+\ab_{j}+\ab_{j+1},
\]
and hence
\[
\ab_i+\ab_j\geq \ab_{i-1}+\ab_{j+1}
\]
as required. The proof of (b) follows on the similar lines as (a).

For the proof of (c) we first show that  $\ab_i+\ab_j>\ab_{i-1}+\ab_{j+1}$ if $\ab_r$ is a corner point for some $r$ with $i<r<j$. Suppose we have equality. Then $\ab_i-\ab_{i-1}=\ab_{j+1}-\ab_j$. On the other hand, by (a) it follows that
\[
\ab_i-\ab_{i-1}\geq \ab_{i+1}-\ab_{i}\geq \ab_{i+2}-\ab_{i+1}\geq \ldots \geq \ab_{j+1}-\ab_j.
\]
Since $\ab_i-\ab_{i-1}=\ab_{j+1}-\ab_j$, we must have equality everywhere in this chain of inequalities. In particular, we have $\ab_r-\ab_{r-1}\geq \ab_{r+1}-\ab_{r}$, and this means that $2\ab_r =\ab_{r-1}+\ab_{r+1}$. This is a contradiction, since $\ab_r$ is a corner point.

In the general case we have
\[
\ab_i+\ab_j>\ab_{i-1}+\ab_{j+1}\geq\ab_{i-2}+\ab_{j+2}\geq \cdots \geq \ab_{i-k}+\ab_{j+k},
\]
as desired.
\end{proof}

Let $\ab$ and $\bb$ be two integer vectors in $\ZZ^2_{\geq 0}$. The line segment $[\ab,\bb]$ between $\ab$ and $\bb$ is defined to be the set $$\{t\ab+(1-t)\bb\;\:\ 0\leq t\leq 1\}.$$

\begin{Lemma}
\label{equi}
Let $\mathcal{A}\: \ab_1,\ldots,\ab_m$ be a concave or a convex sequence,   and let  $\{\ab_{j_1},\ldots,\ab_{j_l}\}$ with $1=j_1<j_2<\ldots<j_{\ell}=m$ be the set of corner points of $\mathcal{A}$.
\begin{enumerate}
\item [\em(a)] For all $j=1,\ldots,m$, there exists an  integer $k$ such that $\ab_j\in [\ab_{j_k},\ab_{j_{k+1}}]$.
\item [\em (b)]  $\ab_{j+1}-\ab_j=\ab_{j_{k+1}}-\ab_{j_k}$ for all $j$ with $j_k\leq j\leq j_{k+1}-1$. In other words, the vectors $\ab_j\in [\ab_{j_k},\ab_{j_{k+1}}]$ are in equidistant position.
\end{enumerate}
\end{Lemma}

\begin{proof}
(a) There exists $k$ such that $j_k\leq j\leq j_{k+1}$. Since $\ab_j$ is not a corner point, it follows that $\ab_j$ belongs to the  line segment $[\ab_{j_k},\ab_{j_{k+1}}]$.

(b) We may assume that there exists $j$ with $j_k<j<j_{k+1}$, otherwise the statement is trivial. Since $j$ is not a corner point, it follows that $2\ab_j=\ab_{j-1}+\ab_{j+1}$, that is, $\ab_j-\ab_{j-1}=\ab_{j+1}-\ab_j$. This holds for all $j$ with $j_k<j<j_{k+1}$. Thus the assertion follows.
\end{proof}

\section{Concave and convex monomial ideals in $K[x,y]$}

Let $K$ be a field and $S=K[x,y]$ be the polynomial ring over $K$ in two indeterminates,  and let $I\subset S$ be a monomial ideal. Let $G(I)=\{u_1,\ldots,u_m\}$ be the unique minimal set of monomial generators of $I$.  Throughout this paper we will always assume that the generators of $I$ are labeled such that $u_1>u_2>\ldots >u_m$ with respect to the lexicographic order. Let $u_i=x^{a_i}y^{b_i}$ for $i=1,\ldots,m$. Then  we have
\[
a_1>a_2>\ldots >a_m\quad \text{and}\quad  b_1<b_2<\ldots <b_m.
\]
The exponent vector of $u_i$ is the vector $\cb_i=(a_i,b_i)$.

Furthermore,  we will always assume that $\height I=2$, because if $\height I=1$, then there exists $f\in S$ such that $I=fJ$, where $\height J=2$. Thus any nonzero ideal in $S$ is isomorphic,  as an  $S$-module,  to a height $2$ ideal in $S$. The condition $\height I=2$ is equivalent to saying that  $a_m=b_1=0$.

\medskip
We call the monomial ideal $I$ concave, resp.\  convex,  if the exponent vectors of the monomial generators of $I$, ordered lexicographically,  form a  concave, resp.\ convex sequence.

\begin{Proposition}
\label{concavehilbert}
 Assume that $I\subset K[x,y]$ is a concave monomial ideal. Then
\begin{enumerate}
\item[{\em (a)}]  $I^k=J^{k-1}I=JI^{k-1}$ for all $k\geq 2$, where $J=(u_1,u_m).$

\item[{\em (b)}]  $\mu(I^k)=(m-1)k+1$ for all $k\geq 0$. In particular, $$\Hilb_{F(I)}(t)=(1+(m-2)t )/(1-t)^2.$$
\end{enumerate}
\end{Proposition}
\begin{proof}
(a)   Since the sequence $\cb_1,\ldots,\cb_m$ of exponent vectors of $I$ is concave, it follows that the sequences $a_1,a_2,\ldots,a_m$ and $b_1,b_2, \ldots,b_m$ are concave as well. In other words, we have $2a_i\geq a_{i-1}+a_{i+1}$ and  $2b_i\geq b_{i-1}+b_{i+1}$ for $i=2,\ldots,m-1$.  It is shown in  \cite[Proposition 4.2]{HMZ} that  $I^2=JI$. Then by induction on $k$ one obtains that $I^k=J^{k-1}I=JI^{k-1}$.

(b) Since $I^k=J^{k-1}I$, it follows that $G(I^k)\subset \Union_{i=1}^{k}u_1^{k-i}u_{m}^{i-1}\{u_1,\ldots,u_m\}$. Notice that for each $i$, the sets $u_1^{k-i}u_{m}^{i-1}\{u_1,\ldots,u_m\}$ and $u_1^{k-i-1}u_{m}^{i}\{u_1,\ldots,u_m\}$ have the monomial $u_1^{k-i}u_m^i$ in common. Therefore,
\begin{eqnarray}
\label{minimal1}
\{u_1^{k},u_1^{k-1}u_2,\ldots,u_1^{k-1}u_m\}\union \Union_{i=2}^k\{u_1^{k-i}u_m^{i-1}u_2,\ldots,  u_1^{k-i}u_m^{i-1}u_m\}.
\end{eqnarray}
is a set of generators of $I^k$.  We claim that this is a minimal set of generators of $I$ which we call  $\mathcal{S}$. This then yields the desired formula for $\mu(I^k)$.
Indeed, let ${u}=u_1^{k-1}u_{\ell}$ and $v=u_1^{k-i}u_m^{i-1}u_j$ be two monomials in $\mathcal{S}$ with $2\leq i\leq k$. Then after canceling the common factor $u_1^{k-i}$ from $u$ and $v$, we need to show that  $ u'=u_1^{i-1}u_{\ell}$ and $v'=u_m^{i-1}u_j$ do not divide each other. Indeed, by comparing coefficients of $x$ and $y$ in $u'$ and $v'$, we see that this is the case  because $a_1(i-1)+a_{\ell}>a_j$ and $b_j+b_m(i-1)> b_{\ell}$.

It remains to show that if $u=u_1^{k-i}u_m^{i-1}u_{\ell}$ and $v=u_1^{k-j}u_m^{j-1}u_{t}$, then $u$ and $v$ can not divide each other, unless $u=v$. We may assume that $i\leq j$. If $i=j$, then after canceling the common factor $u_1^{k-i}u_m^{i-1}$ in $u$ and $v$ we obtain $u_\ell$ and $u_t$ which do not divide each other,  because they are minimal generators of $I$. Thus we may now assume that $i<j$ .

Now we cancel the common factor $u_1^{k-j}u_m^{i-1}$  from $u$ and $v$, and  it remains to show that $u_1^{j-i}u_{\ell}=x^{a_1(j-i)+a_{\ell}}y^{b_\ell}$ and $u_m^{j-i}u_t=x^{a_t}y^{b_m(j-i)+b_t}$ do not divide each other. Then, again by comparing coefficients of $x$ and $y$ in both monomials, we see that they can not divide each other because, $a_1(j-i)+a_{\ell}>a_t$ and $b_t+b_m(j-i)>b_\ell$.

In conclusion, we see that  $\mathcal{S}$ is indeed a minimal set of generators for $I^k$.
\end{proof}

Let $I\in K[x,y]$ be a concave or convex ideal with $G(I)=\{u_1,\ldots,u_m\}$,  and let $\cb_i$ be the exponent vector of $u_i$. We call $\cb_i$ a {\it{corner point}} of $I$, if and only if $\cb_i$ is a corner point of the sequence $\cb_1,\ldots,\cb_m$, as defined in Section~\ref{sequences}.

\begin{Examples}
\label{corners}
{\em  The monomial  ideal $I$ with $$G(I)=\{x^{10}, x^{9}y^2,  x^8y^4, x^7y^5,x^6y^6,x^5y^7,x^4y^8,x^2y^9,y^{10}\},$$ is concave. It  has four corner points, namely $\cb_1=(10,0)$, $\cb_3=(8,4)$, $\cb_7=(4,8)$ and $\cb_9=(0,10)$, as can be seen in Figure~\ref{cornerpoints}.
\begin{figure}[htbp]
\includegraphics[width = 12cm]{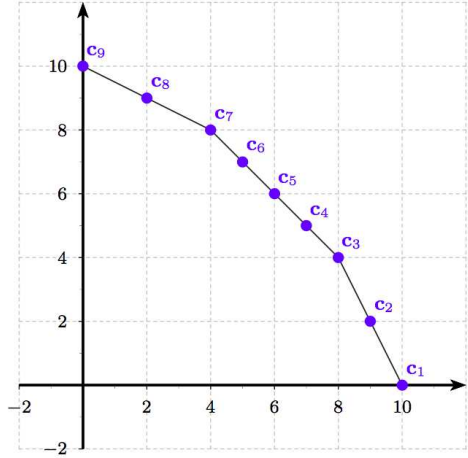}
\caption{The exponent sequence of $I$ in Example~\ref{corners}}\label{cornerpoints}
\end{figure}
The exponent vectors of the  ideal $I$  lie on the  line segments   $[\cb_1,\cb_3]$, $[\cb_3,\cb_7]$ and $[\cb_7,\cb_9]$.
}
\end{Examples}

The line segments whose end points are the corner points of an arbitrary concave or convex ideal $I$, are called the {\em line segments} of $I$.

\begin{Proposition}
\label{regular}
Let $I\subset K[x,y]$ be a concave or convex ideal.  Then $I$ has  no inner corner points, if and only if $I=(x^a,y^b)^k$ for some integers $a,b,k\geq 1$.
\end{Proposition}

\begin{proof}
Suppose $I=(x^a,y^b)^k$, and let $\ab=(a,0)$ and $\bb =(0,b)$.  Then the vectors $\cb_{i+1}=(k-i)\ab+i\bb$ for  $i=0,\ldots,k$ are the  exponent vectors of $I$. Thus it can be seen that $2\cb_i=\cb_{i-1}+\cb_{i+1}$ for $1<i<k+1$. Hence  $I$ has no inner corner points.

 Conversely, suppose that $I$ has no inner corner points and that $\mu(I)=k+1$. Then all exponent vectors of $I$ belong to the line segment $[k\ab,k\bb]$.  By Lemma~\ref{equi}, the exponent vectors of $I$ are in equidistant position, that is $\cb_{i+1}=(k-i)\ab+i\bb$ for  $i=0,\ldots,k$. It follows that $I=(x^a,y^b)^k$.
\end{proof}

\begin{Example}
{\em Let $I$ be the concave monomial ideal of Example~\ref{corners}  and  $I'$ the monomial ideal with $G(I')=\{x^7,x^6y^4,x^5y^7,x^3y^{10},y^{13}\}$. Then the ideal $I'$ is also concave, but $II'$  is not concave. Indeed, $II'$ is minimally generated by $x^{17} > x^{16}y^2 > x^{15}y^4>x^{14}y^5 >x^{13}y^6>x^{12}y^7> x^{11}y^8> x^9y^9> x^7y^{10}> x^6y^{14}> x^5y^{17}>x^3y^{20}> x^2y^{22} >y^{23}$ ordered lexicographically,  and $2\cb_{12}<\cb_{11}+\cb_{13}$, where $\cb_{11}=(5,17)$, $\cb_{12}=(3,20)$ and $\cb_{13}=(2,22)$.}
\end{Example}

The next result shows that for a concave ideal $I$ with an inner corner point, non of  the powers $I^k$ for $k\geq 2$ is concave.

\begin{Proposition}
\label{concavepowers}
Let $I\subset K[x,y]$ be a concave ideal. Then the following conditions are equivalent:
\begin{enumerate}
\item[{\em (a)}] $I$ has no inner corner points;
\item[{\em (b)}] $I^k$ is concave for all integers $k\geq 1$;
\item[{\em (c)}] there exists an integer $k\geq 2$ such that $I^k$ is concave.
\end{enumerate}
In particular, it follows that $I^k$ is not concave for all $k\geq 2$, if  $I$ has an  inner corner point.
\end{Proposition}

\begin{proof}
(a)\implies (b): Since $I$ has no inner corner points, it follows from Proposition~\ref{regular} that $I=(x^a,y^b)^l$ for some integers $a,b,l>0$.  Then $I^k=(x^a,y^b)^{kl}$ is again concave for all $k\geq 1$.

(b)\implies (c) is trivial.

(c)\implies (a): We show that if $I$ has an inner corner point and $k\geq 2$,  then $I^k$ is not  concave.

Let $G(I)=\{u_1,\ldots,u_m\}$ with $u_1>u_2>\ldots>u_m$ in the lexicographical order, as always, and let $v_1=u_1^{k-1}u_{m-1}$, $v_2=u_1^{k-1}u_{m}$ and $v_3=u_1^{k-2}u_2u_m$. Then $v_1,v_2,v_3\in G(I^k)$, as can be seen in the proof of Proposition~\ref{concavehilbert}, equality (\ref{minimal1}). Moreover, $v_1>v_2>v_3$,  and there is no $v\in G(I^k)$ such that $v_1>v>v_2$ or $v_2>v>v_3$.

Assume that $I^k$ is concave, then, considering the exponent vectors of the $v_i$, we obtain the inequality
\[
2(a_1(k-1),b_m)\geq (a_1(k-1)+a_{m-1},b_{m-1})+(a_2+a_1(k-2),b_2+b_m),
\]
from which one easily deduces that $b_m\geq b_{m-1}+b_2$ and $a_1\geq a_{m-1}+a_2$. In other words,
\[
\cb_1+\cb_m\geq \cb_2+\cb_{m-1}.
\]
This is a contradiction, because,  since $I$ is concave and has an inner corner point, it follows from Lemma~\ref{seq}(c) that  $\cb_1+\cb_m < \cb_2+\cb_{m-1}.$ \end{proof}

Now we come to the first main result of this section.

\begin{Theorem}\label{concave1}
Let $I$ be a concave ideal with  $G(I)=\{u_1,\ldots,u_m\}$ and fiber cone  $F(I)=K[z_1,\ldots,z_m]/L$.  Assume that  $$u_1>u_2>\cdots >u_m$$ with respect to the lexicographical order.  Let $u_i=x^{a_i}y^{b_i}$ for $i=1,\ldots,m$, and  set $\cb_i=(a_i,b_i)$ for $i=1,\ldots,m$. Further assume that $\{\cb_{j_1},\ldots,\cb_{j_l}\}$ with $1=j_1<j_2<\ldots<j_{\ell}=m$ is the set of corner points of $I$. Then we have:
 \begin{enumerate}
 \item [\em(a)]  Consider the set $\mathcal B$ of binomials which is the union of the $2$-minors of the matrices
\[
   M_j=
  \left[ {\begin{array}{ccc}
   z_{j_k} &\ldots&z_{j_{k+1}-1} \\
   z_{j_k+1} & \ldots&z_{j_{k+1}} \\
\end{array} } \right],
\quad k=1,\ldots,\ell.
\]
Then $\mathcal B\subset L$.
\item[\em(b)] Let $<$ be the reverse lexicographic order induced by $z_1>\ldots >z_m$, and let $\mathcal M$ be the set of monomials $z_iz_j$ with the property that $1<i\leq j<m$ and $z_iz_j\neq \ini_<(f)$ for any $f\in \mathcal B$. Then $\mathcal M\subset L$.
\item[\em(c)] Let $L_0$ be the ideal generated by the binomials  of $\mathcal B$ together with the monomials  of $\mathcal M$. Then, with respect to the reverse lexicographic order $<$ as defined in {\em (b)}, we have $\ini_<(L_0)=(z_2,\ldots,z_{m-1})^2$.
\item[\em(d)] $L_0=L$.
\item[\em(e)] $F(I)$ is a  Cohen--Macaulay Koszul algebra.
\end{enumerate}
\end{Theorem}

\begin{proof}
(a) Let $f$ be a $2$-minor of one of the matrices listed in (a). Then $f$ is of the form $z_{{j_k}+(r-1)}z_{{j_k}+s}-z_{{j_k}+(s-1)}z_{{j_k}+r}$ with $1\leq r<s\leq j_{k+1}-j_k$.

Since the  vectors $\cb_i\in [\cb_{j_k},\cb_{j_{k+1}}]$ are in equidistant position, it follows that $$\cb_{{j_k}+(r-1)}+\cb_{{j_k}+s}=\cb_{{j_k}+(s-1)}+\cb_{{j_k}+r},$$
and this implies that $f\in L$.

(b) Let $1<i,j<m$. It follows from (a) that $z_iz_j\not\in \ini(f)$ for some $f\in \mathcal{B}$, if and only if one of the following conditions hold:
\begin{enumerate}
\item[(i)] There exist integers $k$ and $l$ with $k\neq l$ such that $$j_k\leq i\leq j_{k+1}\quad \text{and}\quad j_l\leq j\leq j_{l+1}.$$
\item[(ii)] There exist integers $k$ such that $j_k\leq i\leq j_{k+1}$,  and $j=j_k$ or $j=j_{k+1}$.
\end{enumerate}

We must show that   $z_iz_j\in L$,   if $1<i,j<m$ and  $i$ and $j$ satisfy the condition (i) or (ii).

Notice that in general
\begin{eqnarray}
\label{condition}
z_iz_j\in L \quad \text{if and only if} \quad \cb_i+\cb_j>\cb_r+\cb_s
\end{eqnarray}
for some $r$ and $s$.

\medskip
Suppose first that $i$ and $j$ satisfy condition (i) and $i\leq j$. Since $1<i$ and $j<m$, it follows that $1\leq i-1$ and $j+1\leq m$. By (\ref{true}) we have $\cb_i+\cb_j\geq \cb_{i-1}+\cb_{j+1}$. Suppose that equality holds, then $\cb_i-\cb_{i-1}=\cb_{j+1}-\cb_j$. Note that the vector $\cb_i-\cb_{i-1}$ has the same slope as the line segment $\ell_1$ to which $c_i$ and $\cb_{i-1}$ belong. Similarly, the vector $\cb_{j+1}-\cb_j$ has the same slope as the line segment $\ell_2$ to which $\cb_j$ and $\cb_{j+1}$ belong. Since $\ell_1$ is different from $\ell_2$ and all line segments of $I$ have different slop, it follows that $\cb_i+\cb_j\neq \cb_{i-1}+\cb_{j+1}$. Hence, $\cb_i+\cb_j > \cb_{i-1}+\cb_{j+1}$. This shows that $z_iz_j\in L$.

Next suppose that $i$ and $j$ satisfy condition (ii). We may assume that $j=j_{k+1}$ and $j_k\leq i\leq j$. Since $1<i$ and $j<m$, we obtain from Lemma~\ref{seq}(a) the inequality $\cb_i+\cb_{j_{k+1}}\geq \cb_{i-1}+\cb_{j_{k+1}+1}$. Suppose equality holds. Then $\cb_i-\cb_{i-1}=\cb_{j_{k+1}+1}-\cb_{j_{k+1}}$. As before we see that $\cb_i-\cb_{i-1}$ and $\cb_{j_{k+1}+1}-\cb_{j_{k+1}}$ have different slopes. Hence, $\cb_i+\cb_{j_{k+1}}> \cb_{i-1}+\cb_{j_{k+1}+1}$. Again this shows that $z_iz_j\in L$.

(c) From the construction of $\mathcal{B}$ and $\mathcal{M}$ we see that $(z_2,\ldots,z_{m-1})^2 \subseteq \ini_<(L_0)$.

To show the reverse inclusion, we prove that $\mathcal{G}=\mathcal{B} \cup \mathcal{M}$ forms a  Gr\"obner basis of $L_0$ with respect to the reverse lexicographical order. It is well-known that the  $S$-polynomial of any two binomials in $\mathcal{B}$ corresponding to a line segment of $I$  reduces to 0 with respect to reverse lexicographical order. Furthermore, if $f,g \in \mathcal{B}$  belong to different line segments of $I$, then $\gcd(\ini_<(f), \ini_<(g)) =1$,  and hence the $S$-polynomial $S(f,g)$  reduces to $0$. Obviously,  the $S$-polynomials of any two monomials  equals to 0. The only case which needs to be examined is when we consider the $S$-polynomial $S(f,g)$ with $f \in \mathcal{B}$ and $g \in \mathcal{M}$ such that $\gcd(\ini_<(f), g) \neq 1$. As discussed in (a), $f$ is of the form $ z_{{j_k}+(r-1)}z_{{j_k}+s}-z_{{j_k}+(s-1)}z_{{j_k}+r}$ with $1\leq r<s\leq j_{k+1}-j_k$. Then $\ini_<(f) = z_{{j_k}+(s-1)}z_{{j_k}+r}$. Also, from (b), we see that if $g \in \mathcal{M}$ then $g=z_iz_j$ satisfies either condition (i) or (ii).

First, assume that $g$ satisfies condition (i). Then, the condition $\gcd(\ini_<(f), g) \neq 1$ implies that either $z_i$ or $z_j$ is equals to $z_{{j_k}+(s-1)}$ or $z_{{j_k}+r}$. Then, $S(f,g)= z_iz_{{j_k}+(r-1)}z_{{j_k}+s}$ or $S(f,g)= z_jz_{{j_k}+(r-1)}z_{{j_k}+s}$. In both cases, $S(f,g)$ is divided by a monomial in $\mathcal{M}$ satisfying condition (i), and hence reduces to 0 with respect to $\mathcal{G}$.

Next, we assume that $g$ satisfies condition (ii). It is enough to consider the case when $j=j_k$. The case when $j=j_{k+1}$ follows in a similar way. Then $g=z_iz_{j_k}$ with $j_k \leq i\leq j_{k+1}$. Then the condition $\gcd(\ini_<(f), g) \neq 1$ implies that $z_i$ equals to $z_{{j_k}+(s-1)}$ or $z_{{j_k}+r}$. In both cases, $S(f,g) = z_{j_k}z_{{j_k}+(r-1)}z_{{j_k}+s}$,   and  hence $S(f,g)$ is divisible by a monomial in $\mathcal{M}$ satisfying condition (ii). This shows that $S(f,g)$ reduces to  0 with respect $\mathcal{G}$. This gives us $ \ini_<(L_0) \subseteq (z_2,\ldots,z_{m-1})^2$, as required.

(d) Let $B=K[z_1,\ldots,z_{m}]/L_0$ and $C=K[z_1,\ldots,z_{m}]/\ini(L_0)$. By Macaulay's theorem,  $\Hilb_B(t)=\Hilb_C(t)$. By (c) we know that $C= K[z_1,\ldots,z_{m}]/(z_2,\ldots,z_{m-1})^2$:  Therefore, $\Hilb_B(t)=(1+(m-2)t)/(1-t)^2$. Thus Proposition~\ref{concavehilbert} implies that $\Hilb_B(t)=\Hilb_{F(I)}(t)$.

Since $L_0\subseteq L$, there exists a  surjective $K$-algebra homomorphism
$\alpha\: B\to F(I)$, and since the Hilbert function of $B$ coincides with that of $F(I)$, $\alpha$ must be the identity. This shows that $L_0=L$.

(e) The fact  that $F(I)$ is Koszul, follows from Fr\"oberg's theorem \cite{Fr}, since by (c) and (d), $L$ has a quadratic Gr\"obner basis. Moreover, since $C$ is obviously Cohen--Macaulay, it follows  that $F(I)$ (which is $B$) is Cohen--Macaulay as well, see for example \cite[Theorem 3.3.1]{HH}.
\end{proof}

Next we turn to the convex monomial ideals $I\subset K[x,y]$.

\begin{Proposition}
\label{convexhilbert}
 Let  $I\subset K[x,y]$ be  a   convex monomial ideal with  $G(I)=\{u_1,\ldots, u_m\}$ and $u_1>\cdots >u_m$ in the lexicographic order. Then
\begin{enumerate}
\item[{\em (a)}] $I^k=\sum_{i=1}^{m-1}(u_i,u_{i+1})^k$.
\item[{\em (b)}]
\begin{enumerate}
\item [{\em(i)}] $G((u_i,u_{i+1})^k)=\{u_i^{\ell}u_{i+1}^{k-\ell}\:\; \ell=0,\ldots,k\},$
\item [{\em(ii)}] $G(I^k)=\Union_{i=1}^{m-1}G((u_i,u_{i+1})^k),$
\item [{\em (iii)}] for  $1\leq i<j\leq m-1$,  we have $$G((u_i,u_{i+1})^k)\cap G((u_j,u_{j+1})^k)=\begin{cases} \emptyset, \hspace{0.7cm}\text{if  $j\neq i+1$}, & \\ u_{i+1}^k, \text{ if $j=i+1$}. \end{cases}$$
\end{enumerate}
\item[{\em (c)}]  $\mu(I^k)=(m-1)k+1$ for all $k\geq 0$. In particular, $$\Hilb_{F(I)}(t)=(1+(m-2)t )/(1-t)^2.$$
\end{enumerate}
\end{Proposition}

\begin{proof}
(a) It is clear that the ideal $J=\sum_{i=1}^{m-1}(u_i,u_{i+1})^k$ is contained in $I^k$. Conversely, suppose that $u=u_{i_1}u_{i_2}\ldots u_{i_k}\in I^k$ with $i_1\leq i_2\leq \ldots \leq   i_k$ . Let $d(u) = i_k-i_1 $. If $d(u)\leq 1$, then $u\in (u_{i_1},u_{i_1+1})^k$. Now assume that $d(u)> 1$. Let
\[
l(u)=\max\{r \: \; i_1=i_2=\cdots=i_r \text{ and } i_k=i_{k-1}= \cdots=i_{k-r+1}\}.
\]
We claim that there exists $u' \in I$ such that $l(u')=1$ and $u'|u$. Suppose that $l(u) > 1$. Since $I$ is convex, it follows that $\cb_{i+1}+\cb_{j-1} \leq \cb_i +\cb_j$ for $i<j$. This implies that
\begin{equation}\label{divide}
u_{i+1}u_{j-1}| u_{i}u_{j} \quad \text{for all}\quad i<j.
\end{equation}
Let $v=u_{i_1+1}u_{i_2}\ldots u_{i_{k-1}}u_{i_{k}-1}$. Then $l(v) < l(u)$ and $v|u$ by (\ref{divide}). Induction on $l(u)$, completes the proof of the claim.

If we can show that $u'$ with $l(u')=1$ and $u'| u$ belongs to $J$, then $u \in J$, as well. We may assume that $l(u)=1$. Then $v$, as defined before divides $u$ and $d(v)<d(u)$. Therefore, induction on $d(v)$ completes the proof.

(b) Since $I^k=\sum_{i=1}^{m-1}(u_i,u_{i+1})^k$, it   follows that the monomials $u_i^{\ell}u_{i+1}^{k-\ell}$ with  $i=1,\ldots,m-1$ and $\ell=0,\ldots,k$ generate $I^k$.  Suppose that $u_i^{\ell}u_{i+1}^{k-\ell}=vu_j^{t}u_{j+1}^{k-t}$ for some monomial $v$. Then we have
\[
lc_i +(k-l) c_{i+1} \geq tc_j +(k-t) c_{j+1}.
\]
Then after dividing by $k$ we obtain
\[
(l/k)c_i +(1-l/k) c_{i+1} \geq (t/k)c_j +(1-t/k) c_{j+1}.
\]
This means that $(l/k)c_i +(1-l/k) c_{i+1}$ is a point on the line segment with end points  $c_i$ and $c_{i+1}$, and $(t/k)c_j +(1-t/k) c_{j+1}$  is a point on the line segment with end points  $c_j$ and $c_{j+1}$. Since, $I$ is a convex ideal, these points can be equal only if $v=1$ and either
\begin{enumerate}
\item[{(1)}] $i=j$ and $\ell=t$ or,
\item[{(2)}] $i+1=j$, $\ell=0$ and $t=k$ or,
\item[{(3)}] $j+1=i$, $t=0$ and $\ell=k$.
\end{enumerate}
This proves (i), (ii) and (iii).

\medskip
(c)  is an immediate consequence of (b).
\end{proof}

\medskip
The product of two convex ideals is not necessarily convex. Let $$I=(x^{12},x^8y,x^5y^2,x^2y^3,y^4)\quad \text{and} \quad J=(x^{13},x^7y^2,x^3y^6,xy^{12},y^{20}).$$  Then
$$G(IJ)=\{ x^{25},x^{21}y,x^{18}y^2,x^{15}y^3, x^{12}y^4,x^9y^5, x^7y^6, x^5y^9 ,x^3y^{10}, xy^{16}, y^{24}\}.$$ Thus  $IJ$ is not a convex ideal, since $2\cb_8> \cb_7+\cb_9$, where $\cb_7=(7,6)$, $\cb_8=(5,9)$ and $\cb_9=(3,10)$.

\medskip
However, in contrast to Proposition~\ref{concavepowers},  we have

\begin{Proposition}
\label{convexpowers}
Let $I\subset K[x,y]$ be a convex ideal. Then $I^k$ is convex for all $k\geq 1$.
\end{Proposition}

\begin{proof}
Let $\cb_1,\ldots,\cb_m$ the convex sequence of exponent vectors of $G(I)$. Then,  obviously, the sequence $k\cb_1,\ldots,k\cb_m$ is convex as well, and by Proposition~\ref{convexhilbert} the vectors  $k\cb_j$  are exponent vectors of $G(I^k)$. Proposition~\ref{convexhilbert} also implies that the  other exponent vectors of $G(I^k)$ lie on the line segments with end points $k\cb_j$ and $k\cb_{j+1}$. This proves the assertion.
\end{proof}

Let $I\subset K[x,y]$ be a monomial ideal  with  $G(I)=\{u_1,\ldots,u_m\}$ and $u_1>u_2>\cdots >u_m$ with respect to the lexicographical order. We assume that $u_1$ is a pure power of $x$ and $u_m$ a pure power of $y$. We let $J=(u_1,u_m)$. In Proposition~\ref{concavehilbert} we have seen that for a concave ideal that $I^2=JI$. This means that $J$ is a reduction ideal of $I$ and that $r_J(I)=1$. In the next section, where we study symmetric ideals, it is shown that for any integer $a\geq 1$,  there exists a  symmetric ideal $I$ minimally generated by $4$ elements for which $J$ is a reduction ideal and for which $r_J(I)=a$.

The convex ideals behave completely different. Here we have

\begin{Proposition}
\label{convexpowers1}
Let $I\subset K[x,y]$ be a convex ideal with an inner corner point. Then $I^{k+1}\neq JI^k$ for all $k\geq 1$.
\end{Proposition}

\begin{proof}
Let $\cb_{i+1}$ be an inner corner point of $I$ and $u_{i+1}$ be the monomial whose exponent vector is $\cb_{i+1}$. Then $1\leq  i<m-1$. We claim that $u_{i+1}^{k+1}\in I^{k+1}\setminus JI^{k}$ for all $k\geq 1$. In order to prove this we must show that $u_{i+1}^{k+1}$ cannot be written as $u_1v$ or as $u_mv$ with $v\in I^k$. By symmetry it is enough to show that  $u_{i+1}^{k+1}$ cannot be written as $u_1v$.  For the proof of this we use that fact, proved in  Proposition~\ref{convexhilbert},  that $G(I^{r})=\Union_{j=1}^{m-1}G((u_j,u_{j+1})^{r})$ for all $r\geq 1$

Suppose that $u_{i+1}^{k+1}=u_1v$ with $v\in I^k$. Then there exits a number $j$ such that $u_{i+1}^{k+1}\in u_1(u_j,u_{j+1})^k$.
By using (\ref{divide}) we see that $u_1(u_j,u_{j+1})^k\subset \sum_{t\leq j}(u_t,u_{t+1})^{k+1}$. Therefore, there exists a monomial $w$ and integers $l$ and $t$ with $0\leq l\leq k+1$ and $0\leq t\leq j$ such that $u_{i+1}^{k+1}=wu_t^lu_{t+1}^{k+1-l}$. Since the monomials $u_{i+1}^{k+1}$ and $u_t^lu_{t+1}^{k+1-l}$ belong to $G(I^{k+1})$, this equation is only possible, if $w=1$, $l=0$ and $t=i$. In particular, $i\leq j$.

If $j=i$, then $u_{i+1}^{k+1}=u_1u_i^lu_{i+1}^{k-l}$ for some $l$ with $0\leq l\leq k$. By repeated application of (\ref{divide}) we see that  $u_1u_i^lu_{i+1}^{k-l}$ is divided by   $u_s^au_{s+l}^b$ with $s\leq i$, $a+b=k+1$ and $b<k+1$, and hence $u_{i+1}^{k+1}\neq u_1u_i^lu_{i+1}^{k-l}$, a contradiction.

If $j>i$, then $u_{i+1}^{k+1}=u_1u_j^lu_{j+1}^{k-l}$ for some $l$ with $0\leq l\leq k$. We may assume that $l<k$ if $j=i+1$, because otherwise  $u_1=u_{i+1}$, a contradiction. Thus $u_j^lu_{j+1}^{k-l}$ contains a factor $u_p$, where $p=j$ or $p=j+1$,  with $p>i+1$. Since $\cb_{i+1}$ is a corner point of $I$, it follows from  Lemma~\ref{seq}(c) that $u_1u_p\in \mm I^2$. Therefore, $u_1u_j^lu_{j+1}^{k-l}\subset \mm I^{k+1}$. Since  $u_{i+1}^{k+1}\in G(I^{k+1})$, we conclude that  $u_{i+1}^{k+1}\neq u_1u_j^lu_{j+1}^{k-l}$, a contradiction.
\end{proof}

The second  main result of this section is

\begin{Theorem}\label{convex2}
Let $I$ be a convex ideal with  $G(I)=\{u_1,\ldots,u_m\}$ and fiber cone  $F(I)=K[z_1,\ldots,z_m]/L$.  Assume that  $$u_1>u_2>\cdots >u_m$$ with respect to the lexicographical order.  Let $u_i=x^{a_i}y^{b_i}$ for $i=1,\ldots,m$, and  set $\cb_i=(a_i,b_i)$ for $i=1,\ldots,m$. Further assume that $\{\cb_{j_1},\ldots,\cb_{j_l}\}$ with $1=j_1<j_2<\ldots<j_{\ell}=m$ is the set of corner points of $I$. Then we have:
 \begin{enumerate}
 \item [\em(a)]  Consider the set $\mathcal B$ of binomials which is the union of the $2$-minors of the matrices
\[
   M_j=
  \left[ {\begin{array}{ccc}
   z_{j_k} &\ldots&z_{j_{k+1}-1} \\
   z_{j_k+1} & \ldots&z_{j_{k+1}} \\
  \end{array} } \right],
\quad k=1,\ldots,\ell.
\]
Then $\mathcal B\subset L$.
\item[\em(b)] Let $<$ be the lexicographic order induced by $z_1>\ldots >z_m$, and let $\mathcal M$ be the set of monomials $z_iz_j$ with the property that $1\leq i< j-1\leq m-1$,  and $z_iz_j\neq \ini_<(f)$ for any $f\in \mathcal B$. Then $\mathcal M\subset L$.
\item[\em(c)] Let $L_0$ be the ideal generated by the binomials  of $\mathcal B$ together with the monomials  of $\mathcal M$. Then, with respect to the lexicographic order $<$ as defined in {\em (b)}, we have $\ini_<(L_0)=(z_iz_j\:\; 1\leq i< j-1\leq m-1)$.
\item[\em(d)] $L_0=L$.
\item[\em(e)] $F(I)$ is a  reduced Cohen--Macaulay Koszul algebra.
\end{enumerate}
\end{Theorem}

\begin{proof}
(a) The proof of this statement is same as the proof of (a) in Theorem~\ref{concave1}.

(b) We have $z_iz_j =\ini_<(f)$ for some $f \in \mathcal B$ if and only if there exists some integer $k$ such that $j_k \leq i < j-1 \leq j_{k+1}-1$. Therefore, $z_iz_j \in \mathcal{M}$ if and only if $1 \leq i < j-1 \leq m-1$ and there is no $k$ such that $i,j \in [ j_k, j_{k+1}]$.

Now we show that $\mathcal{M} \subset L$. As before in (\ref{condition}) we notice that
\begin{eqnarray}
z_iz_j\in L \quad \text{if and only if} \quad \cb_i+\cb_j>\cb_r+\cb_s
\end{eqnarray}
for some $r$ and $s$.

Let $z_iz_j \in \mathcal{M}$. Then $1 \leq i < j-1 \leq m-1$,  and it follows that $1< i+1$ and $j-1< m$. By Lemma~\ref{seq},  we have $\cb_{i+1}+\cb_{j-1}\leq \cb_{i}+\cb_{j}$. Suppose that equality holds, then $\cb_{i+1}-\cb_{i}=\cb_{j}-\cb_{j-1}$. Note that the vector $\cb_{i+1}-\cb_i$ has the same slope as the line segment $\ell_1$ to which $\cb_i$ and $\cb_{i+1}$ belong. Similarly, the vector $\cb_{j}-\cb_{j-1}$ has the same slope as the line segment $\ell_2$ to which $\cb_j$ and $\cb_{j-1}$ belong. Since $i$ and $j$ belong to different line segment, we conclude that $\ell_1$ is different from $\ell_2$. Also, we know that all line segments of $I$ have different slop, hence it follows that $\cb_{i+1}-\cb_{i} \neq \cb_{j}- \cb_{j-1}$. Consequently, $\cb_{i+1}+\cb_{j-1} < \cb_{i}+\cb_{j}$ which shows that $z_iz_j\in L$.

(c) From the construction of $\mathcal B$ and $\mathcal{M}$, we see that $(z_iz_j\:\; 1\leq i< j-1\leq m-1) \subset \ini_<(L_0)$. To prove the reverse inclusion, we show that $\mathcal G= \mathcal B \cup \mathcal M$ forms a Gr\"obner basis of $L_0$ with respect to the lexicographical order. As discussed in Theorem~\ref{concave1} (c), we see that the only case to be examined is when we consider the $S$-polynomial $S(f,g)$ with $f \in \mathcal{B}$ and $g \in \mathcal{M}$ such that $\gcd(\ini_<(f), g) \neq 1$. As discussed in (a), $f$ is of the form $ z_{{j_k}+(r-1)}z_{{j_k}+s}-z_{{j_k}+(s-1)}z_{{j_k}+r}$ with $1\leq r<s\leq j_{k+1}-j_k$. Then $\ini_<(f) = z_{{j_k}+(r-1)}z_{{j_k}+s}$. Also, from (b), we see that if $g \in \mathcal{M}$ then $g=z_iz_j$ with $1 \leq i < j-1 \leq m-1$ and there is no $k$ such that $i,j \in [ j_k, j_{k+1}]$. The condition $\gcd(\ini_<(f), g) \neq 1$ implies that either $z_i$ or $z_j$ equals to $z_{{j_k}+(r-1)}$ or $z_{{j_k}+s}$. Then $S(f,g)= z_i z_{{j_k}+(s-1)}z_{{j_k}+r}$ or $S(f,g)= z_j z_{{j_k}+(s-1)}z_{{j_k}+r}$. In both cases, $S(f,g)$ is divided by a monomial in $\mathcal{M}$, and hence reduces to 0 with respect to $\mathcal{G}$.

(d) By the same argument as in proof of (d) in Theorem~\ref{concave1}, it is enough to show that $B= K[z_1, \ldots, z_n]/\ini_<(L_0)$ has same Hilbert function as $F(I)$. Therefore, because of Proposition~\ref{convexhilbert}, we have to show that $\Hilb_B(t) = (1+(m-2)t)/(1-t)^2$.

We first show the Alexander dual of $L_0$ is the ideal $L_0^\vee= (v_1,\ldots,v_{m-1})$, where $v_j=(\prod_{i=1}^m z_i)/(z_jz_{j+1})_{j=1,\ldots,m-1}$.
Indeed, let $ S\subset [n]$ and $P_S=(x_i\:\;i\in S)$. We have to show that $P_S$ is a minimal prime ideal of $L_0$ if and only if $S=[n]\setminus\{k, k+1\}$.

Notice that $L_0\subset P_S$ if and only if $S\cap \{i,j\}\neq \emptyset$ for all $\{i,j\}\in \mathcal{T}$, where $\mathcal{T}=\{\{i,j\}\:\; 1\leq i< j-1\leq m-1\}$.

It is clear that $[n]\cap \{i,j\}\neq \emptyset$. Now let $S=[n]\setminus\{k\}$ for some $k$, and let  $\{i,j\}\in \mathcal{T}$, so $S\cap \{i,j\}\neq \emptyset$, because otherwise $\{i,j\}\subset \{k\}$, a contradiction. Finally, assume that at least two elements of $[n]$ do not belong to $S$ and take two of them say $i,j$ with $i<j$. Assume that $j\neq i+1$. Then $\{i,j\}\in \mathcal{T}$, and hence $S\cap \{i,j\}= \emptyset$. This means that $L_0$ is not contained in $P_S$. In other words, if $P_S$ contains $L_0$ it can contain only two elements $i,j$ in the complement and we must have $j=i+1$. This shows that $P_S$ is a minimal prime ideal of $L_0$ if and only if $S=[n]\setminus\{k,k+1\}$.

\medskip
Consider the following matrix
\[
   M=
  \left[ {\begin{array}{cccccc}
   z_m & -z_{m-2}&0&\ldots&\ldots&0 \\
   0 & z_{m-1}&-z_{m-3}&\ldots&\ldots&0 \\
   0 & \dots&\ddots&\ddots&\ldots&\vdots \\
   \vdots & \ldots& 0&z_4&-z_2&0 \\
   0 & \ldots&\ldots&0&z_3&-z_1 \\

  \end{array} } \right]
\]
It is easy to check that the columns of the $m-1\times m-2$-matrix $M$ correspond to the relations of $v_1,\ldots,v_{m-1}$, and that maximal minors of this matrix are just elements $v_1,\ldots,v_{m-1}$. By the Hilbert-Burch Theorem~\cite[Theorem 1.4.17]{BH}, it follows that $M$ is a relation matrix of $L_0^\vee$ and that $L_0^\vee$ is Cohen - Macaulay of height 2. The matrix $M$ also shows that $L_0^\vee$ has a linear resolution. By the Eagon--Reiner Theorem~\cite[Theorem 8.1.9]{HH}, it follows that $L_0$ is Cohen-Macaulay ideal which has a 2-linear resolution.

Now we compute $\Hilb_B(t)$. Since $B$ is Cohen-Macaulay of dimension 2, there exist regular sequence $\ell_1, \ell_2$ of linear forms. Let $\bar{B}=B/(\ell_1, \ell_2)B$. Then $\bar{B}$ is $0$ dimensional standard graded $K$-algebra of embedding dimension $m-2$ whose defining ideal has $2$ - linear resolution. This implies that $\bar{B_i}=0$ for $i\geq 2$. Therefore, $\Hilb_{\bar{B}}(t)= 1+(m-2)t$, and this implies that $\Hilb_B(t)= (1+(m-2)t)/(1-t)^2$, as desired.

(e) By (c) and (d), it follows that $\ini_<(L)$ is generated by squarefree monomial of degree 2. This implies that $L$ is a radical ideal, see \cite[Proposition 3.3.7]{HH}. Moreover, the  Koszul property follows from Fr\"oberg \cite{Fr}. We observed already in proof of (d) that $L_0$ (which is $L$) is a Cohen-Macaulay ideal. This completes the proof of theorem.
\end{proof}

\section{Symmetric ideals}

Let $I$ be a monomial ideal in $S=K[x,y]$ with $G(I)=\{u_1,\ldots,u_m\}$ where $u_i=x^{a_i}y^{b_i}$ with  $a_1>a_2> \cdots >a_m=0$ and $0=b_1<b_2< \cdots <b_m$.
 We call the ideal $I$  a {\em symmetric ideal}, if $b_i=a_{m-i+1}$ for $i=1,\ldots,m$.

 \medskip
 In contrast to convex and concave ideals, the fiber cone of symmetric ideals may not be Cohen-Macaulay and can even have depth $0$. In the following we consider such an example. Let $I$ be the symmetric ideal with
 \begin{eqnarray}\label{symm}
 (a_1,\dots,a_m)  = (5m, 4m, 4m-1, \dots,3m+4,m,0)\quad\text{with} \quad m\geq 5.
 \end{eqnarray}
 In this case, the ideal $I$ is generated in two different degrees, and one can check that $I^2$ is generated in the single degree $10m$, see \cite[Proposition 4.1]{EHM}.

\begin{Proposition}
\label{shalom}
Let $I$ be the  symmetric ideal as defined in (\ref{symm}). Then $$\depth F(I)=0.$$
\end{Proposition}
\begin{proof}
Let $J=(x^{5m},x^{4m}y^m,x^my^{4m},y^{5m})$ and $L$ be the ideal which is generated by the rest of the generators of $I$. So $I=J+L$. Let $u\in L$, which is a generator in degree $7m+3$. Obviously, $L\subset I$ and hence $LI\subset I^2$. The degree of generators of $uI$ is $12m+3$ and $14m+6$. As it mentioned before, degree of generators of $I^2$ is $10m$. Therefore, $uI\subset \mm I^2$. This implies that $ u+\mm I\in \Soc(F(I))$. Moreover, $u+\mm I\neq 0$ since $u$ is a minimal generator of $I$.
\end{proof}

\medskip

In the following, we study in more detail those symmetric ideals $I$ with $\mu(I)=4$. We fix the following notation. Let $0<a<b<c$ be integers with $\gcd(a,b,c)=1$. Then we define the symmetric ideal $I=(x^c, x^by^a, x^ay^b,y^c)$.

\begin{Theorem}\label{application}
The depth of the fiber cone $F(I)$ is $2$ if $2a \leq b$ and $2b \leq a+c$, or $2a \geq b$ and $2b \geq a+c$. In particular, $F(I)$ is Cohen-Macaulay.
\end{Theorem}

\begin{proof}
The inequalities $2a \leq b$ and $2b \leq a+c$, or $2a \geq b$ and $2b \geq a+c$ guarantee that $I$ is either a convex  or a concave ideal. Thus, the result follows from Theorem~\ref{convex2} and Theorem~\ref{concave1}.
\end{proof}

Unfortunately, the conditions given in Theorem~\ref{application} are only sufficient. Consider the example, when $a=3$, $b=5$ and $c=9$. Then $I=(x^9,x^5y^3,x^3y^5,y^9)$ is neither convex nor concave but $\depth F(I)=2$.

\medskip

Next we consider the case, when $I$ is generated in single degree, which is equivalent to say that $c=a+b$. In that case, we have

\begin{Theorem}\label{sum}
The fiber cone $F(I)$ is Cohen-Macaulay if and only if $b=a+1$.
\end{Theorem}

For the proof of this theorem, we use the result given by Cavaliere and Niesi  \cite{CN}: Let $S_1$ be a numerical semigroup with generators $0<n_1<n_2< \ldots < n_d$, and let $a \in S_1$. The Ap\'ery
set $\Ap(a,S_1)$ of $S_1$ with respect to $a$ is defined to be the set
\[
\{s\in S_1 \: \; s-a \notin S_1 \}
\]
It is known that $|\Ap(a,S_1)|=a$.

\medskip

Now let $S \subset \NN^2$ be the semigroup generated by
\[
\{ (0,n_d), (n_1, n_{d}-n_1), \ldots, (n_{d-1}, n_d-n_{d-1}), (n_d,0)\}.
\]
The semigroup ring $K[S]$ of $S$ is the coordinate ring of the projective monomial curve defined by  $0<n_1<n_2< \ldots < n_d$. We denote by $S_2$, the numerical semigroup generated by $0< n_d-n_{d-1}<n_{d}-n_{d-2}<\ldots <n_d-n_1<n_d$. Let $B_1=\Ap(n_d,S_1)$ and $B_2=\Ap(n_d,S_2)$. Let $B_1=\{0, \nu_1, \ldots, \nu_{n_d-1}\}$. For each $\nu_i \in B_1$ with $\nu_i \neq 0$, let $\mu_i \in S_2$ be the smallest element such that $(\nu_i, \mu_i) \in S$. Then the criterion given by Cavalieri and Niesi says that the following conditions are equivalent:
\begin{enumerate}
\item[(i)] $K[S]$ is Cohen-Macaulay.
\item[(ii)] $B_2=\{0,\mu_1, \ldots, \mu_{n_d-1}\}$.
\end{enumerate}

\begin{proof}[Proof of Theorem~\ref{sum}.]
In our case $S=\{(0,a+b),(a,b),(b,a),(a+b,0)\}$ and hence $S_1=S_2$ and $B_1=B_2$. Furthermore, $n_d=a+b$. We denote $B_1$ simply by $B$, and claim that
\begin{equation}\label{set}
B=\{0,1a, 2a,\ldots, (b-1)a, 1b, 2b, \ldots ab \}.
\end{equation}
To prove this we first show that $ab \in B$. Indeed, since $\gcd (a,b)=1$, it is known that $ab-(a+b) \notin S_1$. Suppose $ia \notin B$ for some $1\leq i \leq b-1$. Then $ia-(a+b) \in S_1$ which implies that $ab-(a+b)= ia-(a+b) + (b-i)a \in S_1$, a contradiction. Similarly, we see that $jb \in S_1$ for all $1\leq j \leq a-1$.

We claim that the elements in $\{0,1a, 2a,\ldots, (b-1)a, 1b, 2b, \ldots ab \}$ are pairwise distinct. Indeed, suppose that $ia=jb$ for some $1\leq i \leq b-1$ and $1 \leq j \leq a-1$. Since $\gcd(a,b)=1$, it follows that $b|i$ contradicting the fact that $1\leq i\leq b-1$. Hence, (\ref{set}) holds.

Now we apply the criterion of Cavaliere and Niesi to prove the main assertion of the theorem. By this criterion, $F(I)$ is Cohen-Macaulay if and only if for all $\nu \in B$, there exists $\mu \in B$ such that $\mu$ is the smallest element in $S_1$ with the property that $(\nu,\mu) \in S$. Let $ia \in B$ for some $1 \leq i \leq b-1$, and let $\mu \in S_1$ be the smallest element such that $(ia,\mu) \in S$. Then there exist $r_k \geq 0$ for $ k=1, \ldots,4$ such that:
\begin{enumerate}
\item[(i)]$r_1+r_2+r_3+r_4= i$;
\item[(ii)]$(r_2+r_4)a+(r_3+r_4)b=ia$;
\item[(iii)]$(r_1+r_3)a+(r_1+r_2)b=\mu$.
\end{enumerate}
From (ii), we obtain $(r_3+r_4)b=(i-(r_2+r_4))a$, and from (i) we see that $r_2+r_4 \leq i$. Therefore, $d:=i-(r_2+r_4)\geq 0$. Since $\gcd(a,b)=1$, it follows that $b|d$. Since $1 \leq i \leq b-1$, it implies $0\leq d\leq i < b$, and hence $d=0$. In other words, $i=r_2+r_4$. By (ii), we deduce that $r_3+r_4=0$, hence $r_3=r_4=0$ and $r_2=i$. Now (i) implies that $r_1=0$. Then (iii) implies that $\mu=ib$. Similarly, we can see that for $jb \in B$ with $1 \leq j \leq a$, $ja$ is the smallest element in $S_1$ such that $(jb,ja) \in S$. Therefore, $F(I)$ is Cohen-Macaulay if and only if $|\{1a,2a,\ldots, (b-1)a\}|=|\{1b,2b,\ldots, ab\}|$, which is satisfied if and only $b=a+1$, as required.
\end{proof}

\medskip

The criterion of Cavaliere and Niesi does not apply to ideals which are not generated in a single degree, that is, when $c \neq a+b$. For example, let $a=3$, $b=4$ and $c=6$. Then for $I=(x^6,x^4y^3,x^3y^4,y^6)$, the fiber cone $F(I)$ is Cohen-Macaulay, but does not satisfy the criterion of Cavaliere and Niesi. Indeed, let $S_1$ be the numerical semigroup generated by $3,4$ and $6$. Then $B=\Ap(6,S_1)=\{0,3,4,7,8,11\}$, and $(8,6)$ is the only element in $S$ which is of the form $(8,\mu)$, but $6\notin B$.
\medskip

Suppose $F(I)$ is Cohen-Macaulay. Then, under the assumption of  Theorem~\ref{sum}, the reduction number can be determined.

\begin{Theorem}
\label{reduction}
The ideal $J=(x^{2a+1},y^{2a+1})$ is a minimal reduction ideal of
\[
I=(x^{2a+1},x^{a+1}y^a,x^ay^{a+1},y^{2a+1}),
\]
and the reduction number $r_J(I)$ of $I$ with respect to $J$ is equal to $a$.
\end{Theorem}

\begin{proof}
We first prove  that $I^{a+1}=JI^a$. This then shows that $J$ is a reduction ideal of $I$ and that  $r_J(I)\leq a$.

Let $L=(x^{a+1}y^a,x^ay^{a+1})$. Then $I=J+L$, and we have to show that $L^{a+1}\subset JL^a$. Let $v\in G(L^k)$ for some $k\geq 1$. Then there exist integers $r,s\geq 0$ with $r+s=k$,  and such that $v=(x^{a+1}y^a)^r(x^ay^{a+1})^s$. Thus
\[
G(L^k)=\{x^{ak+r}y^{ak+s}\:\; r,s\geq 0 \text{ and } r+s=k\}.
\]
Suppose now that $v\in G(L^{a+1})$. Then $v=x^{a(a+1)+r}y^{a(a+1)+s}$ with $r,s\geq 0$ and $r+s=a+1$. Let us first assume that $r>0$ and $s>0$. Then
\[
v=x^{2a+1}y^{2a+1}(x^{a(a-1)+r-1}y^{a(a-1)+s-1}).
\]
Since $x^{a(a-1)+r-1}y^{a(a-1)+s-1}\in L^{a-1}$, it follows that $v\in J^2L^{a-1}\subset JI^a$.

Next consider the case when $r=0$ or $s=0$. We may assume that $r=0$. Then $v=x^{a(a+1)}y^{(a+1)^2}=y^{2a+1}(x^{a+1}y^a)^a$. This shows that $v\in JL^a\subset JI^a$, as desired.

\medskip
It remains to be shown that  $r_J(I)\geq a$. Let $u=x^ay^{a+1}$. The desired inequality will follow once we have shown that $u^k\in I^k\setminus JI^{k-1}$ for all $k<a$. Suppose that $u^k\in JI^{k-1}$ for some $k<a$. Then there exists $i>0$ such that $x^{ak}y^{(a+1)k}\in J^iL^{k-i}$.  Hence there exist integers $p,q\geq 0$ with $p+q=i$ and integers $r,s\geq 0$ with $r+s=k-i$ such that
\[
x^{ak}y^{(a+1)k}=(x^{p(2a+1)}y^{q(2a+1)})(x^{a(k-i)+r}y^{a(k-i)+s}).
\]
This leads to the equation $ak=a(k-i)+r+p(2a+1)$ which implies that $p+r=a(i-2p)$. This implies that $a$ divides $p+r$. Since $0\leq p+r \leq k<a$, this is possible only if $p+r=0$ and $i=2p$. From $p+r=0$ we deduce that $p=0$. But then also $i=0$, a contradiction.
\end{proof}

For a general symmetric ideal $I=(x^c,x^by^a,x^ay^b,y^c)$, $F(I)$ may be Cohen-Macaulay without having $J=(x^c,y^c)$ as a minimal reduction ideal. For example, $(x^8,y^8)$ is not a minimal reduction ideal of $I=(x^8,x^4y^3,x^3y^4,y^8)$. Indeed, it can be checked that when $k$ is even, then $x^{(k/2)7}y^{(k/2)7} \in I^k \setminus JI^{k-1}$, and when $k$ is odd, then $x^{((k-1)/2)7+3}y^{((k-1)/2)7+4} \in I^k \setminus JI^{k-1}$.

\begin{Theorem}\label{bigc}
Let $0<a<b<c$ be integers,  and let $I=(x^c,x^by^a,x^ay^b,y^c)$.
\begin{enumerate}
\item[{\em (a)}]  If $2a\geq c$, then
\begin{enumerate}
\item[{\em (i)}] $F(I) \iso K[z_1,z_2,z_3,z_4]/ (z_2z_3, z_2^2,z_3^2)$,
\item[{\em (ii)}] $\Hilb_{F(I)}(t) = (1+ 2t)/ (1-t)^2$,
\item[{\em (iii)}] $F(I)$ is Cohen-Macaulay.
\end{enumerate}

\item[{\em (b)}]
If such that $c > r (b-a) + a$ where $r=\left \lceil{b/(b-a)}\right \rceil$, then
\begin{enumerate}
\item[{\em (i)}] $F(I) \iso K[z_1,z_2,z_3,z_4]/ (z_1z_3^{r-1},z_2^{r-1}z_4,z_1z_4)$,
\item[{\em (ii)}] $\Hilb_{F(I)}(t) = (1+ 2\sum_{i=2}^{r-1} t^i)/ (1-t)^2$,
\item[{\em (iii)}] $F(I)$ is Cohen-Macaulay.
\end{enumerate}
\end{enumerate}
\end{Theorem}

\begin{proof}
(a) Let $F(I) = S/L$,  where $S=K[z_1,z_2,z_3,z_4]$ and $L=\Ker (S \longrightarrow R(I))$ is defined by $z_1\mapsto u_1=x^c$, $z_2\mapsto u_2=x^by^a$, $z_3\mapsto u_3=x^ay^b$ and $z_4\mapsto u_4=y^c$. For  the Rees ring $R(I)$ of $I$ we have the relations:
\[
z_2^2-x^{2b-c}y^{2a-c}z_1z_4, \quad z_3^2-x^{2a-c}y^{2b-c}z_1z_4, \quad  z_2z_3-x^{a+b-c}y^{a+b-c}z_1z_4,\ldots
\]
It  follows that $z_2^2,z_2z_3, z_3^2\in L$. From now on the arguments are exactly the same as those in Proposition~\ref{concavehilbert} to prove that $L=(z_2^2,z_2z_3, z_3^2)$. Once we have this,  it is obvious that $F(I)$ is Cohen--Macaulay.

(b) We first show that $L_0= (z_1z_3^{r-1},z_2^{r-1}z_4,z_1z_4) \subset L$. Let $K[x,y,z_1,z_2,z_3,z_4]/J$ represent  the Rees ring $R(I)$ of $I$. From $c > r (b-a) + a$, it follows that $c > b+a$, and then $z_1z_4 - (xy)^{c-(b+a)}z_2z_3 \in J$. This shows that $z_1z_4 \in L$.

By the choice of $r$, it follows that $r(b-a)-b \geq 0$, and hence
\[
z_1z_3^{r-1} - x^{c-(r(b-a)+a)} y^{r(b-a)-b} z_2^r \in J.
\]
This shows that $z_1z_3^{r-1} \in L$. Similarly, we have $z_2^{r-1}z_4 \in L_0$. Therefore, $L_0 \subset L$ and there is a natural $K$-algebra homomorphism $A=S/L_0 \longrightarrow F(I)$. In order to show  that this is an isomorphism, we show that $\Hilb_{A}(t) =\Hilb_{F(I)}(t) $. Note that $L_0$ is the ideal of $2$-minors of matrix

\[
  \left[ {\begin{array}{ccc}
   z_{1} &z_2^{r-1}&0 \\
   0 & z_3^{r-1}&z_4 \\
\end{array} } \right].
\]
By Hilbert-Burch \cite[Theorem 1.4.17]{BH} it follows that $L_0$ is a height 2 Cohen-Macaulay ideal and
\[
0\longrightarrow S^2(-r-2)\longrightarrow S^2(-r)\dirsum S(-2)\longrightarrow S \longrightarrow A\longrightarrow 0
\]
is the graded free $S$-resolution of $A$. This shows that
\[
\Hilb_A(t)=(1-t^2-2t^r+2t^{r+1})/(1-t)^4= (1+ 2\sum_{i=2}^{r-1} t^i)/ (1-t)^2.
\]
Therefore, (i) and  also (iii) follows from (ii).

\medskip
(ii) Let $I_1=(x^c,x^by^a,x^ay^b)$, $I_2=(x^by^a,x^ay^b,y^c)$ and $M=(x^by^a,x^ay^b)$. We first show that
\begin{enumerate}
\item[{(1)}] $|G(I_1^k) |=|G(I_2^k)| = { k+2\choose 2}$, for $k=0,\ldots, r-1$,
\item[{ (2)}]  $|G(I_1^k) |=|G(I_2^k)| = { k+2\choose 2}-{ k-r+ 2\choose 2}$ for $k \geq r$,
\item[{ (3)}] $|G(M^k)| = k+1$,
\item[{ (4)}] $G(I_1^k) \cap G(I_2^k)= G(M^k)$
\item[{ (5)}] $G(I^k)= G(I_1^k) \cup G(I_2^k)$
\end{enumerate}
imply that $\Hilb_A(t)= (1+ 2\sum_{i=2}^{r-1} t^i)/ (1-t)^2$.

Indeed,  (1),(3), (4) and (5)  imply that $\mu(I^k) = 2 { k+2\choose 2}-(k+1)=(k-1)^2$ for $ k \leq r-1$.  By using (1),(2), (4) and (5) one  shows that   $\mu(I^r)=r^2+2r-1$, and ´that  $\mu(I^{k+1})- \mu(I^k) =2r-1$ for all $k\geq r$. This yields
\[
\mu(I^k) =(2r-1)(k-r-1)+r^2 \quad \text{for}\quad k\geq r.
\]

In conclusion we get
\begin{eqnarray*}
\Hilb_{F(I)}(t)&=& \sum_{k \geq 0} \mu(I^k) t^k = \sum_{k=0}^{r-1} (r+1)^2 t^k + \sum_{k \geq r} ((2r-1)(k-r+1)+r^2)t^k  \\
&=& (1+ 2\sum_{i=2}^{r-1} t^i)/ (1-t)^2,
\end{eqnarray*}
as desired.

Proof of (1) and (3): It is enough to prove these statements for $I_1$. By symmetry they  then follow also for $I_2$.  In order to  prove (i) and (ii)   we show that  $F(I_1)\iso K[z_1,z_2,z_3]/(z_1z_3^{r-1})$. This isomorphism of standard graded $K$-algebras then obviously implies the desired identities for $|G(I_1^k)|$.

Note that $u_1,u_2$ as well as $u_2,u_3$ are algebraically independent. Therefore, a generating relation of the Rees ring $R(I_1)$ of $I_1$ is of the form
$h=v_1z_1^iz_3^j-v_2z_2^{i+j}$, where $v_1$ and $v_2$ are monomials in $K[x,y]$. Any relation of $F(I_1)$ is obtained from such a relation by reduction modulo $(x,y)$. Thus the nonzero relations of $F(I_1)$ are induced by relations of the form $h$,  where either $v_1=1$ or §´$v_2=1$.  Since $c>a+b$, it follows that $\deg u_1^iu_3^j=ic+j(a+b)>(i+j)(a+b)=\deg u_2^{i+j}$, unless $i=0$. In the latter  case $h$ is of the form $v_1z_3^j-v_2z_2^j$ with $v_1\neq 1\neq v_2$, and the induced relation for $F(I_1)$ is trivial. Thus $h$ induces a nonzero relation for $F(I_1)$ if and only if $i>0$. In this case $v_1=1$  and $v_2$ is monomial in $K[x,y]$ of degree $i(c-(i+j))>0$. The discussion shows that any relation of $J_1$ is of the form $z_1^iz_3^j$ with $i>0$. At the begin of the proof of the theorem  we have seen that  $z_1z_3^{r-1}$  belongs to $F(I_1)$. We now show that if $z_1^iz_3^j\in J_1$,  then $j\geq r-1$. Since $i>0$, this then implies that $z_1z_3^{r-1}$  divides $z_1^iz_3^j$,  yielding the desired conclusion.

Indeed, if $z_1^iz_3^j$ with $i>0$ belongs to $J_1$, then there exists  a relation $h=z_1^iz_3^j-v_2z_2^{i+j}$ of $F(I_1)$ with $v_2\in (x,y)$. This implies that $u_1^iu_3^j=v_2u_2^{i+j}$ Comparing the exponents of $y$ on both sides we see that $jb\geq (i+j)a$. From this we obtain that
$(1+j/i)(b-a)\geq b$, and hence $1+j/i\geq r$. It follows that $j\geq j/i\geq r-1$,  as  desired.

\medskip
Proof of (3)  and (4): From  the proof of (1) and (2) we obtain as a side result that
\begin{eqnarray}\label{one}
\hspace{0.5cm} G(I_1^k)&=& \{u_1^{k_1}u_2^{k_2}u_3^{k_3}\: \; k_1+k_2+k_3=k \text{ and $k_1=0$,  or $k_1>0$ and $k_3<r$}\},
\end{eqnarray}
and
\begin{eqnarray}\label{two}
\hspace{0.5cm} G(I_2^k)&= &\{u_2^{k_2}u_3^{k_3}u_4^{k_4}\: \; k_2+k_3+k_4=k \text{ and $k_4=0$,  or $k_4>0$ and $k_2<r$}\}.
\end{eqnarray}
From this description of $G(I_1^k)$ and $G(I_2^k)$ it follows immediately that $G(I_1^k)\sect G(I_2^k)=G(M^k)$. This proves (4). Statement (3) follows from the fact that $u_2$ and $u_3$ are algebraically independent.

Proof of (v): First we show that $I^k=I_1^k+I_2^k$. It is clear that $I_1^k+I_2^k \subseteq I^k$. For the other inclusion, it is enough to show that if $u\in G(I^k)$ then $u \in I_1^k+I_2^k$. Note that $u_1u_4 \in \mm I^2$. So, if $u=u_{i_1}\ldots u_{i_k}$ and $1,4 \in \{i_1, \ldots, i_k\}$, then $u \in \mm I^k$ and $u \notin G(I^k)$, a contradiction. Therefore, if $1 \in \{i_1, \ldots, i_k\}$, then $u \in I_1^k$, otherwise $u \in I_2^k$. Hence $I^k=I_1^k+I_2^k$. In particular, we have $G(I^k) \subseteq G(I_1^k) \cup G(I_2^k)$.

Suppose that $G(I^k) \subsetneq G(I_1^k) \cup G(I_2^k)$. Then there exists $u \in G(I_1^k) \cup G(I_2^k)$ such that $u \notin G(I^k)$. We may assume that $u \in G(I_1^k)$. Since $u \notin G(I^k)$, there exists $u' \in G(I_2^k)$ such that $u'=wu$ for some monomial $w \neq 1$. We show by induction on $k$ that it is not possible. The assertion is clear for $k=1$. Let $k \geq 2$. Then there exists integers $0\leq i,j\leq k$, and monomials $v \in M^{k-i}$ and $v' \in M^{k-j}$ such that  $u=u_1^iv$ and $u'=u_4^j v'$. If $i=0$, then $u,u'\in G(I_2^k)$, and hence they cannot divide each other, and if $j=0$,   then $u,u'\in G(I_1^k)$, and again the cannot divide each other. Thus we may assume that $i,j>0$.
Furthermore, we have the equation $u_1^iv=wu_4^j v'$. Since $v \in M^{k-i}$ and $v' \in M^{k-j}$, we have $v=u_2^ru_3^s$ and $v'=u_2^{r'}u_3^{s'}$ with $r+s=k-i$ and $r'+s'=k-j$. If $r,r'>0$ or $s,s'>0$, then we can cancel the common factor from both sides of $u_1^iv=wu_4^j v'$ and get the desired result by induction. Otherwise, we have the following two cases

\begin{eqnarray}\label{8}
 v=u_2^{k-i}, v'=u_3^{k-j},
  \end{eqnarray}
  or
  \begin{eqnarray}\label{9}
v=u_3^{k-i}, v'=u_2^{k-j}.
\end{eqnarray}
Since $ ic+(a+b)(k-i) =\deg u > \deg u'= jc+(a+b)(k-j)$ and $c > a+b$, it follows that $i>j$. By comparing the exponents of $y$ in case (\ref{8}), we get $a(k-i) > jc+b(k-j)$ which is not possible. 

Now we consider the case (\ref{9}).  By comparing the exponents of $y$ in case (\ref{9}), we get
\begin{eqnarray}
b(k-i) > jc+a(k-j).
\end{eqnarray}
Then by using $i>j$, we obtain $b(k-i) > jc+a(k-i)$ and therefore,
$(b-a)(k-i) > jc$. Recall that $r=\left \lceil{b/(b-a)}\right \rceil$. Then, since $i>0$, it follows form (\ref{one}  that $r>k-i$. Therefore,  $ b/(b-a)> k-i$, and so  and $b> (b-a)(b/(b-a))>jc$ which is impossible since $j>0$.
\end{proof}

\begin{Corollary}
\label{shift}
Let $m\geq 0$ be an integer, and let $I\subset K[x,y]$ be the symmetric ideal attached to the sequence $0<a+m<b+m <c+m$ of integers. Then $F(I)$ is Cohen--Macaulay for all $m\geq c-2a$.
\end{Corollary}

\begin{proof}
If $m\geq c-2a$, then $2(a+m)\geq c+m$. Thus the assertion follows from Theorem~\ref{bigc}(a).
\end{proof}

The following corollary just sums up what we proved on Theorem~\ref{bigc}.

\begin{Corollary}
\label{interval}
Let $I$ be the symmetric ideal attached to the sequence of integers $0<a<b<c$. Then $F(I)$ is Cohen--Macaulay for all $c\not \in [2a+1,r(b-a)+a]$, where $r=\lceil b/(b-a)\rceil$.
\end{Corollary}

The behavior of $\depth F(I)$ for $c\in [2a+1,r(b-a)+a]$ with $c\neq a+b$ seems to be hard to predict.
For example,  $F(I)$ is Cohen--Macaulay for $a=5$, $b=7$ and $c=11$, while $F(I)$ is not Cohen--Macaulay if  $a=2$, $b=7$,and $c=8$.  In both cases, $2a+1<c<a+b$.
In the next examples,  $a+b< c< r(b-a)+a$ and $F(I)$ is Cohen--Macaulay for
 $a=7$, $b=13$ and  $c=23$, while  $F(I)$ is not Cohen--Macaulay  for $a=7$, $b=13$ and  $c=24$.

\medskip
However we have

\begin{Corollary}
\label{nice}
Let $I$ be the symmetric ideal attached to the sequence $0<a<b<c$. If $b=a+1$, then $F(I)$ is Cohen--Macaulay for all $c$.
\end{Corollary}

\begin{proof}
If $b=a+1$, then $2a+1=a+b$ and $r(b-a)+a=a+b$, and hence $[2a+1,r(b-a)+a]=\{a+b\}$.  The assertion follows from Theorem~\ref{sum} and Corollary~\ref{interval}.
\end{proof}

We finally would like to remark that in all examples that we considered we had $\depth F(I)>0$ for the $4$-generated symmetric ideals. Unfortunately, at present we cannot prove  this in general. We should mention that for symmetric ideals with $5$  or more generators,  one very well  may have $\depth F(I)=0$, as Proposition~\ref{shalom} shows.

\end{document}